\documentclass{article}
\usepackage{hyperref}
\usepackage{cite}
\usepackage{amsfonts}
\usepackage{amsthm}
\usepackage{amsmath}
\usepackage{amssymb}
\usepackage{mathtools} 
\usepackage{authblk} 
\usepackage{xcolor}
\usepackage{soul} 
\usepackage{fontenc}

\usepackage[top=1in, bottom=1in, left=1in, right=1in]{geometry}

\hypersetup{colorlinks, linkcolor=red, filecolor=darkgreen, urlcolor=blue, citecolor=blue}
\theoremstyle{plain}
\newtheorem{theorem}{Theorem}[section]
\newtheorem{lemma}{Lemma}[section]
\newtheorem{proposition}{Proposition}[section]
\newtheorem{corollary}{Corollary}[section]

\theoremstyle{definition}
\newtheorem{definition}{Definition}[section]

\numberwithin{equation}{section}

\definecolor{pomcol}{rgb}{0,0,0.8}
\definecolor{afcol}{rgb}{1,0,0}

\makeatletter
\newcommand{\subjclass}[2][]{%
  \let\@oldtitle\@title%
  \gdef\@title{\@oldtitle\footnotetext{#1 \emph{Mathematics subject classification (2010):} #2}}%
}
\newcommand{\keywords}[1]{%
  \let\@@oldtitle\@title%
  \gdef\@title{\@@oldtitle\footnotetext{\emph{Keywords:} #1}}%
}
\makeatother

\begin{document}


\title{Fractional integral inequalities of Hermite-Hadamard type for convex functions with respect 
to a monotone function}
\date{}
\author[1]{Pshtiwan Othman Mohammed\thanks{Corresponding author. Email: \texttt{pshtiwansangawi@gmail.com}}}
\affil[1]{{\small Department of Mathematics, College of Education, University of Sulaimani, Sulaimani, 
Kurdistan Region, Iraq}}

\keywords{Fractional integral; Convex functions; Hermite-Hadamard inequality} 
\subjclass{26D07; 26D15; 26D10; 26A33}

\maketitle

\begin{abstract}
In the literature, the left-side of Hermite--Hadamard's inequality is called a midpoint type inequality. In 
this article, we obtain new integral inequalities of midpoint type for Riemann--Liouville fractional integrals 
of convex functions with respect to increasing functions. The resulting inequalities generalize some recent 
classical integral inequalities and Riemann--Liouville fractional integral inequalities established in earlier 
works. Finally, applications of our work are demonstrated via the known special functions of real numbers.
\end{abstract}
\section{Introduction}
\label{sec:intro}
A function $g:\mathcal{I}\subseteq\mathbb{R}\to\mathbb{R}$ is said to be convex on the
interval $\mathcal{I}$, if the inequality
\begin{align}\label{eq:11}
g(\eta\,x+(1-\eta)y)\leq \eta\,g(x)+(1-\eta)g(y)
\end{align}
holds for all $x,y\in\mathcal{I}$ and $\eta\in[0,1]$.  We say that $g$ is concave, provided $-g$ is convex.

For convex functions \eqref{eq:11}, many equalities and inequalities have been established, {\em e.g.},
Ostrowski type inequality \cite{7}, Opial inequality \cite{Farid}, Hardy type inequality \cite{8}, Olsen type 
inequality \cite{9}, Gagliardo-Nirenberg type inequality \cite{10}, midpoint and trapezoidal type inequalities 
\cite{6,Mohammed9} and the Hermite--Hadamard type (HH-type) inequality \cite{5} that will be used in our study, 
which is defined by:
\begin{align}\label{eq:12}
g\left(\frac{u+v}{2}\right)&\leq \frac{1}{v-u}\int_u^v g(x)dx\leq \frac{g(u)+g(v)}{2},
\end{align}
where $g:\mathcal{I}\subseteq\mathbb{R}\to\mathbb{R}$ is assumed to be a convex function on  
$\mathcal{I}$ where $a, b\in \mathcal{I}$ with $u<v$.

A huge number of researchers in the field of applied and pure mathematics have devoted their efforts to
modify, generalize, refine, and extend the Hermite--Hadamard inequality \eqref{eq:12} for convex and other 
classes of convex functions; see for further details \cite{5,16,17,18,4}.

In 2013, the HH-type inequality \eqref{eq:12} has been generalised to fractional integrals of Riemann--Liouville 
type by Sarikaya et al \cite{1}. Their result is as follows, for an $L^1$ convex function 
$f:[u,v]\to\mathbb{R}$, and for any $\mu>0$:
\begin{align}\label{eq:13}
g\left(\frac{u+v}{2}\right)&\leq \frac{\Gamma(\mu+2)}{2(v-u)^\mu}\left[I^{\mu}_{u^+}g(v)
+I^{\mu}_{v^-}g(u)\right]\leq \frac{g(u)+g(v)}{2},
\end{align}
where $I^{\mu}_{u^+}$ and $I^{\mu}_{v^-}$ denote left-sided and right-sided 
Riemann-Liouville fractional integrals of order $\mu>0$, respectively, defined as~\cite{11}:
\begin{align}\label{eq:14}
\begin{aligned}
I^{\mu}_{u^+}g(x)=\frac{1}{\Gamma(\mu)}\int_u^x (x-t)^{\mu-1}g(t)dt, \quad x>u,
\\
I^{\mu}_{v^-}g(x)=\frac{1}{\Gamma(\mu)}\int_x^v (t-x)^{\mu-1}g(t)dt, \quad x<v.
\end{aligned}
\end{align}
If we take $\mu=1$ in \eqref{eq:13} we obtain \eqref{eq:12}, it is clear that inequality \eqref{eq:13} 
is a generalization of Hermite--Hadamard inequality \eqref{eq:12}.
Many further results have been derived from this \cite{14,24,Mohammed6,Mohammed7,Mohammed4,2}, 
including in different types of fractional calculus, e.g. for tempered fractional integrals \cite{Mohammed10}, 
those of Hilfer type \cite{basci-baleanu}, for those models of fractional calculus involving Mittag-Leffler 
kernels \cite{Fernandez-Mohammed}, and for fractional integrals with respect to functions \cite{Mohammed5}. 
But so far such inequalities have not been investigated for fractional integrals of a twice differentiable
convex function with respect to a monotone function. For this reason, we recall the Riemann--Liouville 
fractional integrals of a function with respect to a monotone function.
\begin{definition}\label{def:11}
Let $(u,v)\subseteq(-\infty,\infty)$ be a finite or infinite interval of the real-axis $\mathbb{R}$ and 
$\mu>0$. Let $\psi(x)$ be an increasing and positive monotone function on the interval $(u,v]$ with a 
continuous derivative $\psi'(x)$ on the interval $(u,v)$. Then the left and right-sided 
$\psi$-Riemann--Liouville fractional integrals of a function $g$ with respect to another function $\psi(x)$ 
on $[u,v]$ are defined by \cite{11,19,20}:
\begin{align}\label{eq:15}
\begin{aligned}
I^{\mu:\psi}_{u^+}g(x)=\frac{1}{\Gamma(\mu)}\int_{u}^{x} \psi'(t)(\psi(x)-\psi(t))^{\mu-1}g(t)dt,
\\
I^{\mu:\psi}_{v^-}g(x)=\frac{1}{\Gamma(\mu)}\int_{x}^{v} \psi'(t)(\psi(t)-\psi(x))^{\mu-1}g(t)dt.
\end{aligned}
\end{align}
It is important to note that if we set $\psi(x)=x$ in \eqref{eq:15}, then $\psi$-Riemann--Liouville 
fractional integral reduces to Riemann--Liouville fractional integral \eqref{eq:14}.
\end{definition}

As we said, in this study we investigate several inequalities of midpoint type for Riemann--Liouville 
fractional integrals of twice differentiable convex functions with respect to increasing functions.
\section{Main Results}
Our main results follow the following lemma:
\begin{lemma}\label{lem:31}
Let $g:[u,v]\subseteq\mathbb{R}\to\mathbb{R}$ be a differentiable function and $g''\in L_1[u,v]$ with 
$0\leq u<v$. If $\psi(x)$ is an increasing and positive monotone function on $(u,v]$ and its derivative 
$\psi'(x)$ is continuous on $(u,v)$, then for $\mu\in(0,1)$ we have
\begin{align}\label{eq:31}
\begin{aligned}
\sigma_{\mu,\psi}(g;u,v)&=\frac{2^{\mu-1}}{(v-u)^{\mu}}\Biggl[
\int_{\psi^{-1}\left(\frac{u+v}{2}\right)}^{\psi^{-1}(v)}\psi'(t)(v-\psi(t))^{\mu+1}(g''\circ \psi)(t)dt
\\
&-\int_{\psi^{-1}(u)}^{\psi^{-1}\left(\frac{u+v}{2}\right)} \psi'(t)(\psi(t)-u)^{\mu+1}
(g''\circ \psi)(t)dt\Biggr],
\end{aligned}
\end{align}
where
\begin{align*}
\sigma_{\mu,\psi}(g;u,v)&=\frac{2^{\mu-1}\Gamma(\mu+2)}{(v-u)^\mu}
\left[I^{\mu:\psi}_{\psi^{-1}\left(\frac{u+v}{2}\right)^+}
(g\circ \psi)\left(\psi^{-1}(v)\right)+I^{\mu:\psi}_{\psi^{-1}\left(\frac{u+v}{2}\right)^-}
(g\circ \psi)\left(\psi^{-1}(u)\right)\right]
-(\mu+1)g\left(\frac{u+v}{2}\right).
\end{align*}
\end{lemma}
\begin{proof}
From Definition \ref{def:11} we have
\begin{align*}
\hbar_{1}:=\frac{2^{\mu-1}\Gamma(\mu+2)}{(v-u)^\mu}
I^{\mu:\psi}_{\psi^{-1}\left(\frac{u+v}{2}\right)^+}(g\circ \psi)\left(\psi^{-1}(v)\right)
&=\frac{\mu(\mu+1)2^{\mu-1}}{(v-u)^\mu}\int_{\psi^{-1}\left(\frac{u+v}{2}\right)}^{\psi^{-1}(v)} 
\psi'(t)(v-\psi(t))^{\mu-1}(g\circ \psi)(t)dt
\\
&=-\frac{(\mu+1)2^{\mu-1}}{(v-u)^\mu}\int_{\psi^{-1}\left(\frac{u+v}{2}\right)}^{\psi^{-1}(v)} 
(g\circ \psi)(t)d(v-\psi(t))^{\mu}.
\end{align*}
Integrating by parts twice, we have
\begin{align}\label{eq:32}
\hbar_{1}&=\frac{\mu+1}{2}g\left(\frac{u+v}{2}\right)+\frac{(\mu+1)2^{\mu-1}}{(v-u)^\mu}
\int_{\psi^{-1}\left(\frac{u+v}{2}\right)}^{\psi^{-1}(v)}\psi'(t)(v-\psi(t))^{\mu}(g'\circ \psi)(t)dt
\nonumber\\
&=\frac{\mu+1}{2}g\left(\frac{u+v}{2}\right)+\frac{1}{2}g'\left(\frac{u+v}{2}\right)
+\frac{(\mu+1)2^{\mu-1}}{(v-u)^\mu}
\int_{\psi^{-1}\left(\frac{u+v}{2}\right)}^{\psi^{-1}(v)}\psi'(t)(v-\psi(t))^{\mu+1}(g''\circ \psi)(t)dt
\end{align}
Analogously
\begin{align}\label{eq:33}
\hbar_{2}&:=\frac{2^{\mu-1}\Gamma(\mu+1)}{(v-u)^\mu}
I^{\mu:\psi}_{\psi^{-1}\left(\frac{u+v}{2}\right)^+}(g\circ \psi)\left(\psi^{-1}(v)\right)
\nonumber\\
&=\frac{\mu+1}{2}g\left(\frac{u+v}{2}\right)-\frac{1}{2}g'\left(\frac{u+v}{2}\right)
-\frac{2^{\mu-1}}{(v-u)^\mu}
\int_{\psi^{-1}(u)}^{\psi^{-1}\left(\frac{u+v}{2}\right)} \psi'(t)(\psi(t)-u)^{\mu+1}(g''\circ \psi)(t)dt.
\end{align}
It follows from \eqref{eq:32} and \eqref{eq:33} that
\begin{align*}
\hbar_{1}+\hbar_{2}-(\mu+1)g\left(\frac{u+v}{2}\right)
&=\frac{2^{\mu-1}}{(v-u)^{\mu}}\Biggl[
\int_{\psi^{-1}\left(\frac{u+v}{2}\right)}^{\psi^{-1}(v)}\psi'(t)(v-\psi(t))^{\mu+1}(g''\circ \psi)(t)dt
\\
&-\int_{\psi^{-1}(u)}^{\psi^{-1}\left(\frac{u+v}{2}\right)} \psi'(t)(\psi(t)-u)^{\mu+1}
(g''\circ \psi)(t)dt\Biggr].
\end{align*}
This completes the proof of Lemma \ref{lem:31}.
\end{proof}
\begin{corollary}\label{cor:31}
With the similar assumptions of Lemma \ref{lem:31} if
\begin{enumerate}
\item $\psi(x)=x$, we have
\begin{align*}
&\frac{2^{\mu-1}\Gamma(\mu+2)}{(v-u)^\mu}\left[I^{\mu}_{\left(\frac{u+v}{2}\right)^+}g(v)
+I^{\mu}_{\left(\frac{u+v}{2}\right)^-}g(u)\right]-(\mu+1)g\left(\frac{u+v}{2}\right)
\\
&=\frac{(v-u)^2}{8}\Biggl[\int_{0}^{1}t^{\mu+1}g''\left(\frac{t}{2}u+\frac{2-t}{2}v\right)dt
+\int_{0}^{1}t^{\mu+1}g''\left(\frac{2-t}{2}u+\frac{t}{2}v\right)dt\Biggr],
\end{align*}
which is obtained by Tomar et al. \cite{21}.
\item $\psi(x)=x$ and $\mu=1$, we have
\begin{align*}
&\frac{1}{v-u}\int_{u}^{v}g(x)dx-g\left(\frac{u+v}{2}\right)
=\frac{(v-u)^2}{16}\Biggl[\int_{0}^{1}t^{2}g''\left(\frac{t}{2}u+\frac{2-t}{2}v\right)dt
+\int_{0}^{1}t^{2}g''\left(\frac{2-t}{2}u+\frac{t}{2}v\right)dt\Biggr],
\end{align*}
\end{enumerate}
which is obtained by Sarikaya and Kiris \cite{22}.
\end{corollary}
\begin{theorem}\label{th:31}
Let $g:[u,v]\subseteq\mathbb{R}\to\mathbb{R}$ be a differentiable function and $g''\in L_1[u,v]$ with 
$0\leq u<v$. Suppose that $|g''|$ is convex on $[u,v]$, $\psi(x)$ is an increasing and positive monotone 
function on $(u,v]$ and its derivative $\psi'(x)$ is continuous on $(u,v)$, then for $\mu\in(0,1)$ we have
\begin{align}\label{eq:34}
\begin{aligned}
\left|\sigma_{\mu,\psi}(g;u,v)\right|
&\leq\frac{(v-u)^2}{8}\left(\frac{1}{\mu+2}\right)^{1-\frac{1}{q}}
\Biggl\{\left[\frac{1}{2(\mu+3)}|g''(u)|^q+\left(\frac{1}{\mu+2}-\frac{1}{2(\mu+3)}\right)
|g''(v)|^q\right]^{\frac{1}{q}}
\\
&+\left[\left(\frac{1}{\mu+2}-\frac{1}{2(\mu+3)}\right)|g''(u)|^q
+\frac{1}{2(\mu+3)}|g''(v)|^q\right]^{\frac{1}{q}}\Biggr\}
\end{aligned}
\end{align}
for $q\geq 1$.
\end{theorem}
\begin{proof}
Suppose that $q=1$. By means of Lemma \ref{lem:31} and Definition \ref{def:11}, we get
\begin{align}\label{eq:35}
\sigma_{\mu,\psi}(g;u,v)&=\frac{2^{\mu-1}}{(v-u)^{\mu}}\Biggl[
\int_{\psi^{-1}\left(\frac{u+v}{2}\right)}^{\psi^{-1}(v)}\psi'(t_{1})(v-\psi(t_{1}))^{\mu+1}
(g''\circ \psi)(t_{1})dt_{1}
\nonumber \\
&-\int_{\psi^{-1}(u)}^{\psi^{-1}\left(\frac{u+v}{2}\right)} \psi'(t_{2})(\psi(t_{2})-u)^{\mu+1}
(g''\circ \psi)(t_{2})dt_{2}\Biggr].
\end{align}
Change the variables $x_{1}=\frac{2(v-\psi(t_{1}))}{v-u}$ and $x_{2}=\frac{2(\psi(t_{2})-u)}{v-u}$ and then
set $t=x_{1}=x_{2}$ into the resulting equality, then \eqref{eq:35} becomes
\begin{align*}
\sigma_{\mu,\psi}(g;u,v)&=\frac{(v-u)^2}{8}\Biggl[
\int_{0}^{1}t^{\mu+1}g''\left(\frac{t}{2}u+\frac{2-t}{2}v\right)dt
+\int_{0}^{1}t^{\mu+1}g''\left(\frac{2-t}{2}u+\frac{t}{2}v\right)dt\Biggr],
\end{align*}
that is
\begin{align}\label{eq:36}
\left|\sigma_{\mu,\psi}(g;u,v)\right|&\leq\frac{(v-u)^2}{8}\Biggl[
\int_{0}^{1}t^{\mu+1}\left(\left|g''\left(\frac{t}{2}u+\frac{2-t}{2}v\right)\right|
+\left|g''\left(\frac{2-t}{2}u+\frac{t}{2}v\right)\right|\right) dt\Biggr].
\end{align}
By using the convexity of $|g''|$, then inequality \eqref{eq:36} gives 
\begin{align*}
\left|\sigma_{\mu,\psi}(g;u,v)\right|&\leq\frac{(v-u)^2}{8}\Biggl[
\int_{0}^{1}t^{\mu+1}\left(\left|g''\left(\frac{t}{2}u+\frac{2-t}{2}v\right)\right|
+\left|g''\left(\frac{2-t}{2}u+\frac{t}{2}v\right)\right|\right) dt\Biggr]
\nonumber \\
&\leq\frac{(v-u)^2}{8}\left(|g''(u)|\int_{0}^{1}\frac{1}{2}t^{\mu+2} dt
+|g''(v)|\int_{0}^{1}\frac{2-t}{2}t^{\mu+1} dt \right.
\nonumber \\
&+\left.|g''(v)|\int_{0}^{1}\frac{1}{2}t^{\mu+2} dt
+|g''(u)|\int_{0}^{1}\frac{2-t}{2}t^{\mu+1} dt \right)
\nonumber \\
&=\frac{(v-u)^2}{8(\mu+2)}\left(|g''(u)|+|g''(v)|\right).
\end{align*}
This gives \eqref{eq:34} for $q=1$.

Now, suppose that $q>1$. Using inequality of \eqref{eq:36}, convexity of $|g''|^q$ and the power--mean's 
inequality for $q>1$, we have
\begin{align}\label{eq:314}
\int_{0}^{1}t^{\mu+1}\left|g''\left(\frac{t}{2}u+\frac{2-t}{2}v\right)\right|dt
&=\int_{0}^{1}t^{\mu+1-\frac{\mu+1}{q}}\left[t^{\frac{\mu+1}{q}}
\left|g''\left(\frac{t}{2}u+\frac{2-t}{2}v\right)\right|\right]dt
\nonumber \\
&\leq \left(\int_{0}^{1}t^{\mu+1}\right)^{1-\frac{1}{q}}
\left(\int_{0}^{1}t^{\mu+1}\left|g''\left(\frac{t}{2}u+\frac{2-t}{2}v\right)\right|^{q}dt
\right)^{\frac{1}{q}}
\nonumber \\
&\leq \left(\frac{1}{\mu+2}\right)^{1-\frac{1}{q}}
\left(\int_{0}^{1}\left(\frac{t^{\mu+2}}{2}|g''(u)|^{q}+\frac{2t^{\mu+1}-t^{\mu+2}}{2}|g''(v)|^{q}
\right)dt\right)^{\frac{1}{q}}
\nonumber \\
&=\left(\frac{1}{\mu+2}\right)^{1-\frac{1}{q}}
\left[\frac{1}{2(\mu+3)}|g''(u)|^q+\left(\frac{1}{\mu+2}-\frac{1}{2(\mu+3)}\right)
|g''(v)|^q\right]^{\frac{1}{q}}.
\end{align}
In the same manner, we get
\begin{align}\label{eq:315}
&\int_{0}^{1}t^{\mu+1}\left|g''\left(\frac{2-t}{2}u+\frac{t}{2}v\right)\right|dt
\leq \left(\frac{1}{\mu+2}\right)^{1-\frac{1}{q}}
\left[\left(\frac{1}{\mu+2}-\frac{1}{2(\mu+3)}\right)|g''(u)|^q
+\frac{1}{2(\mu+3)}|g''(v)|^q\right]^{\frac{1}{q}}.
\end{align}
Using \eqref{eq:314} and \eqref{eq:315} in \eqref{eq:36} we obtain \eqref{eq:34} for $q>1$.
Thus the proof of theorem \ref{th:31} is completed.
\end{proof}
\begin{corollary} \label{cor:32}
With the similar assumptions of Theorem \ref{th:31} if
\begin{enumerate}
\item $\psi(x)=x$, we have
\begin{align*}
&\left|\frac{2^{\mu-1}\Gamma(\mu+2)}{(v-u)^\mu}\left[I^{\mu}_{\left(\frac{u+v}{2}\right)^+}g(v)
+I^{\mu}_{\left(\frac{u+v}{2}\right)^-}g(u)\right]-(\mu+1)g\left(\frac{u+v}{2}\right)\right|
\\
&\leq\frac{(v-u)^2}{8}\left(\frac{1}{\mu+2}\right)^{1-\frac{1}{q}}
\Biggl\{\left[\frac{1}{2(\mu+3)}|g''(u)|^q+\left(\frac{1}{\mu+2}-\frac{1}{2(\mu+3)}\right)
|g''(v)|^q\right]^{\frac{1}{q}}
\\
&+\left[\left(\frac{1}{\mu+2}-\frac{1}{2(\mu+3)}\right)|g''(u)|^q
+\frac{1}{2(\mu+3)}|g''(v)|^q\right]^{\frac{1}{q}}\Biggr\},
\end{align*}
which is obtained by Tomar et al. \cite{21}.
\item $\psi(x)=x$ and $\mu=1$, we have
\begin{align*}
&\left|\frac{1}{v-u}\int_{u}^{v}g(x)dx-g\left(\frac{u+v}{2}\right)\right|
\leq\frac{(v-u)^2}{48}\Biggl[\left(\frac{3|g''(u)|^q+5|g''(v)|^q}{8}\right)^{\frac{1}{q}}
+\left(\frac{5|g''(u)|^q+3|g''(v)|^q}{8}\right)^{\frac{1}{q}}\Biggr],
\end{align*}
which is obtained by Sarikaya et al. \cite{23}.
\item $\psi(x)=x$ and $q=1$, we have
\begin{align*}
&\left|\frac{2^{\mu-1}\Gamma(\mu+2)}{(v-u)^\mu}\left[I^{\mu}_{\left(\frac{u+v}{2}\right)^+}g(v)
+I^{\mu}_{\left(\frac{u+v}{2}\right)^-}g(u)\right]-(\mu+1)g\left(\frac{u+v}{2}\right)\right|
\leq\frac{(v-u)^2}{8(\mu+2)}\biggl(|g''(u)|+|g''(v)|\biggr),
\end{align*}
which is obtained by Tomar et al. \cite{21}.
\item $\psi(x)=x, \mu=1$ and $q=1$, we have
\begin{align*}
&\left|\frac{1}{v-u}\int_{u}^{v}g(x)dx-g\left(\frac{u+v}{2}\right)\right|
\leq\frac{(v-u)^2}{24}\left(\frac{|g''(u)|+|g''(v)}{2}\right),
\end{align*}
\end{enumerate}
which is obtained by Sarikaya et al. \cite{23}.
\end{corollary}
\begin{theorem}\label{th:32}
Let $g:[u,v]\subseteq\mathbb{R}\to\mathbb{R}$ be a differentiable function and $g''\in L_1[u,v]$ with 
$0\leq u<v$. Suppose that $|g''|^q$ is convex on $[u,v]$, $\psi(x)$ is an increasing and positive monotone 
function on $(u,v]$ and its derivative $\psi'(x)$ is continuous on $(u,v)$, then for $\mu\in(0,1)$ we have
\begin{align}\label{eq:37}
\begin{aligned}
\left|\sigma_{\mu,\psi}(g;u,v)\right|
&\leq\frac{(v-u)^2}{8}\left(\frac{2}{(\mu+1)p+1}\right)^{\frac{1}{p}}
\Biggl[\left(\frac{|g''(u)|^q+3|g''(v)|^q}{4}\right)^{\frac{1}{q}}
+\left(\frac{3|g''(u)|^q+|g''(v)|^q}{4}\right)^{\frac{1}{q}}\Biggr]
\\
&\leq\frac{(v-u)^2}{8}\left(\frac{2}{(\mu+1)p+1}\right)^{\frac{1}{p}}\bigl(|g''(u)|+|g''(v)|\bigr),
\end{aligned}
\end{align}
such that $q>1$ and $\frac{1}{p}+\frac{1}{q}=1$.
\end{theorem}
\begin{proof}
By using the Holder's inequality, we have
\begin{align}\label{eq:38}
\int_{0}^{1}t^{\mu+1}\left|g''\left(\frac{t}{2}u+\frac{2-t}{2}v\right)\right|dt
&\leq \left(\int_{0}^{1}t^{(\mu+1)p}\right)^{\frac{1}{p}}
\left(\int_{0}^{1}\left|g''\left(\frac{t}{2}u+\frac{2-t}{2}v\right)\right|^{q}dt\right)^{\frac{1}{q}}
\nonumber \\
&\leq \left(\frac{1}{(\mu+1)p+1}\right)^{\frac{1}{p}}
\left(\int_{0}^{1}\left(\frac{t}{2}|g''(u)|^{q}+\frac{2-t}{2}|g''(v)|^{q}\right)
dt\right)^{\frac{1}{q}}
\nonumber \\
&=\left(\frac{1}{(\mu+1)p+1}\right)^{\frac{1}{p}}
\left(\frac{|g''(u)|^{q}+3|g''(v)|^{q}}{4}\right)^{\frac{1}{q}}.
\end{align}
Similarly, we have
\begin{align}\label{eq:39}
\int_{0}^{1}t^{\mu+1}\left|g''\left(\frac{2-t}{2}u+\frac{t}{2}v\right)\right|dt
&\leq \left(\frac{1}{(\mu+1)p+1}\right)^{\frac{1}{p}}
\left(\frac{3|g''(u)|^{q}+|g''(v)|^{q}}{4}\right)^{\frac{1}{q}}.
\end{align}
Thus, the inequalities \eqref{eq:36}, \eqref{eq:38} and \eqref{eq:39} complete the proof of
the first inequality of \eqref{eq:37}.

To prove the second inequality of \eqref{eq:37}, we apply the formula 
\begin{align*}
\sum_{i=1}^{n}\left(c_i+d_i\right)^m\leq \sum_{i=1}^{n}c_i^m+\sum_{i=1}^{n}+d_i^m, \quad 0\leq m<1
\end{align*}
for $c_{1}=3|g''(u)|^{q}, c_{2}=|g''(u)|^{q}, d_{1}=|g''(v)|^{q}, d_{2}=3|g''(v)|^{q}$ and $m=\frac{1}{q}$.
Then \eqref{eq:36} gives
\begin{align*}
\left|\sigma_{\mu,\psi}(g;u,v)\right|
&\leq\frac{(v-u)^2}{8}\left(\frac{1}{(\mu+1)p+1}\right)^{\frac{1}{p}}
\Biggl[\left(\frac{|g''(u)|^q+3|g''(v)|^q}{4}\right)^{\frac{1}{q}}
+\left(\frac{3|g''(u)|^q+|g''(v)|^q}{4}\right)^{\frac{1}{q}}\Biggr]
\\
&\leq\frac{(v-u)^2\left(3^{\frac{1}{q}}+1\right)}{16}\left(\frac{1}{(\mu+1)p+1}\right)^{\frac{1}{p}}
\bigl[|g''(u)|+|g''(v)|\bigr]
\\
&\leq\frac{(v-u)^2}{8}\left(\frac{1}{(\mu+1)p+1}\right)^{\frac{1}{p}}\bigl(|g''(u)|+|g''(v)|\bigr).
\end{align*}
Hence the proof of Theorem \ref{th:32} is completed.
\end{proof}
\begin{corollary}\label{cor:33}
With the similar assumptions of Theorem \ref{th:32}, if
\begin{enumerate}
\item $\psi(x)=x$, we have
\begin{align*}
&\left|\frac{2^{\mu-1}\Gamma(\mu+2)}{(v-u)^\mu}\left[I^{\mu}_{\left(\frac{u+v}{2}\right)^+}g(v)
+I^{\mu}_{\left(\frac{u+v}{2}\right)^-}g(u)\right]-(\mu+1)g\left(\frac{u+v}{2}\right)\right|
\\
&\leq\frac{(v-u)^2}{8}\left(\frac{2}{(\mu+1)p+1}\right)^{\frac{1}{p}}
\Biggl[\left(\frac{|g''(u)|^q+3|g''(v)|^q}{4}\right)^{\frac{1}{q}}
+\left(\frac{3|g''(u)|^q+|g''(v)|^q}{4}\right)^{\frac{1}{q}}\Biggr]
\\
&\leq\frac{(v-u)^2}{8}\left(\frac{2}{(\mu+1)p+1}\right)^{\frac{1}{p}}\bigl(|g''(u)|+|g''(v)|\bigr),
\end{align*}
which is obtained by Tomar et al. \cite{21}.
\item $\psi(x)=x$ and $\mu=1$, we have
\begin{align*}
\left|\frac{1}{v-u}\int_{u}^{v}g(x)dx-g\left(\frac{u+v}{2}\right)\right|
&\leq\frac{(v-u)^2}{16(2p+1)^{\frac{1}{p}}}
\Biggl[\left(\frac{|g''(u)|^q+3|g''(v)|^q}{4}\right)^{\frac{1}{q}}
+\left(\frac{3|g''(u)|^q+|g''(v)|^q}{4}\right)^{\frac{1}{q}}\Biggr]
\\
&\leq\frac{(v-u)^2}{2^{2+\frac{2}{q}}(2p+1)^{\frac{1}{p}}}\bigl(|g''(u)|+|g''(v)|\bigr),
\end{align*}
\end{enumerate}
which is obtained by Sarikaya et al. \cite{23}.
\end{corollary}
\begin{corollary}\label{cor:34}
From Theorems \ref{th:31}--\ref{th:32}, we obtain the following inequality for $\psi(x)=x, \mu=1$ and $q>1$:
\begin{align*}
\left|\frac{1}{v-u}\int_{u}^{v}g(x)dx-g\left(\frac{u+v}{2}\right)\right|
&\leq(v-u)^2\min\{\delta_{1},\delta_{2}\}\bigl(|g''(u)|+|g''(v)|\bigr),
\end{align*}
where $\delta_{1}=\frac{1}{24}$ and $\delta_{2}=\frac{1}{2^{2+\frac{2}{q}}(2p+1)^{\frac{1}{p}}}$
such that $p=\frac{q}{q-1}$.
\end{corollary}
\section{Applications}
In this section some applications are presented to demonstrate usefulness of our obtained results in the
previous sections. 
\subsection{Applications to special means}
Let $u$ and $v$ be two arbitrary positive real numbers, then consider the following special means:
\begin{enumerate}
\item[(i)] The arithmetic mean:
\[A=A(u,v)=\frac{u+v}{2}.\]
\item[(ii)] The inverse arithmetic mean:
\[H=H(u,v)=\frac{2}{\frac{1}{u}+\frac{1}{v}}, \quad u,v\neq 0.\]
\item[(iii)] The geometric mean:
\[G=G(u,v)=\sqrt{u\,v}.\]
\item[(iv)] The logarithmic mean:
\[L(u,v)=\frac{v-u}{\log(v)-\log(u)}, \quad u\neq v.\]
\item[(v)] The generalized logarithmic mean:
\[L_{n}(u,v)=\left[\frac{v^{n+1}-u^{n+1}}{(v-u)(n+1)}\right]^{\frac{1}{n}}, 
\quad n\in\mathbb{Z}\setminus\{-1,0\}.\]
\end{enumerate}
\begin{proposition}\label{prop:1}
Let $|n|\geq 3$ and $u, v\in\mathbb{R}$ with $0<u<v$, then
\begin{align}\label{eq:prop1}
\left|A^{n}(u,v)-L_{n}^{n}(u,v)\right|\leq \frac{(v-u)^2|n(n-1)|}{3\cdot 4^{\frac{1}{q}+2}}\left[
A^{\frac{1}{q}}\left(3|u|^{(n-2)q},5|v|^{(n-2)q}\right)
+A^{\frac{1}{q}}\left(5|u|^{(n-2)q},3|v|^{(n-2)q}\right)\right],
\end{align}
for $q\geq 1$.
\end{proposition}
\begin{proof}
Apply Corollary \ref{cor:32} part (2) for $g(x)=x^n$, where $n$ as specified above.
\end{proof}
\begin{proposition}\label{prop:2}
Let $u, v\in\mathbb{R}$ with $0<u<v$, then
\begin{align}\label{eq:prop2}
\left|A^{-1}(u,v)-L^{-1}(u,v)\right|\leq \frac{(v-u)^2}{3\cdot 4^{\frac{1}{q}+2}}\left[
A^{\frac{1}{q}}\left(3|u|^{-3q},5|v|^{-3q}\right)
+A^{\frac{1}{q}}\left(5|u|^{-3q},3|v|^{-3q}\right)\right],
\end{align}
for $q\geq 1$.
\end{proposition}
\begin{proof}
Apply Corollary \ref{cor:32} part (2) for $g(x)=\frac{1}{x}, x\neq 0$.
\end{proof}
\begin{proposition}\label{prop:3}
Let $|n|\geq 3$ and $u, v\in\mathbb{R}$ with $0<u<v$, then
\begin{align}\label{eq:prop31}
&\left|H^{-n}(v,u)-L_{n}^{n}\left(v^{-1},u^{-1}\right)\right|
\nonumber\\
&\leq \frac{\left(v^{-1}-u^{-1}\right)^2|n(n-1)|}{3\cdot 4^{\frac{1}{q}+2}}\left[
H^{\frac{-1}{q}}\left(3|u|^{(n-2)q},5|v|^{(n-2)q}\right)
+H^{\frac{-1}{q}}\left(5|u|^{(n-2)q},3|v|^{(n-2)q}\right)\right],
\end{align}
and
\begin{align}\label{eq:prop32}
\left|H(v,u)-L^{-1}\left(v^{-1},u^{-1}\right)\right|\leq 
\frac{\left(v^{-1}-u^{-1}\right)^2}{3\cdot 4^{\frac{1}{q}+2}}\left[
H^{\frac{-1}{q}}\left(3|u|^{-3q},5|v|^{-3q}\right)
+H^{\frac{-1}{q}}\left(5|u|^{-3q},3|v|^{-3q}\right)\right],
\end{align}
for $q\geq 1$.
\end{proposition}
\begin{proof}
Observe that $A^{-1}\left(u^{-1},v^{-1}\right)=H(u,v)=\frac{2}{\frac{1}{u}+\frac{1}{v}}$.
So, Make the change of variables $u\to v^{-1}$ and $v\to u^{-1}$ in the inequalities
\eqref{eq:prop1} and \eqref{eq:prop2}, we can deduce the desired inequalities \eqref{eq:prop31} and 
\eqref{eq:prop32} respectively.
\end{proof}
\begin{proposition}\label{prop:4}
Let $u, v\in\mathbb{R}$ with $0<u<v$, then
\begin{align}\label{eq:prop4}
\left|G^{-2}(u,v)-A^{-2}(u,v)\right|\leq \frac{(b-a)^2}{2\cdot 4^{\frac{1}{q}+1}}\left[
A^{\frac{1}{q}}\left(3|u|^{-3q},5|v|^{-3q}\right)
+A^{\frac{1}{q}}\left(5|u|^{-3q},3|v|^{-3q}\right)\right],
\end{align}
for $q\geq 1$.
\end{proposition}
\begin{proof}
Apply Corollary \ref{cor:32} part (2) for $g(x)=x^2$.
\end{proof}

Now, we give an application to a midpoint formula. Let $d$ be a partition
$u=x_{0}<x_{1}<\cdots<x_{m-1}<x_{m}=v$ of the interval $[a,b]$ and consider the quadrature formula
\begin{align*}
\int_{a}^{b}g(x)dx=T(g,d)+E(g,d),
\end{align*}
where
\begin{align*}
T(g,d)=\sum_{j=0}^{m-1}g\left(\frac{x_{j}+x_{j+1}}{2}\right)(x_{j+1}-x_{j})
\end{align*}
is the midpoint version and $E(g,d)$ denotes the associated approximation error. Here, we
present some error estimates for the midpoint formula.
\begin{proposition}\label{prop:5}
Let $g:[u,v]\to\mathbb{R}$ be a differentiable mapping on $(u,v)$ with $u<v$. Suppose that $|g''|^q, q\geq 1$ 
be a convex function, then for every partition of $[u,v]$ the midpoint error satisfies
\begin{align}\label{eq:prop5}
\left|E(g,d)\right|\leq \min\{\delta_{1},\delta_{2}\}\,\sum_{j=0}^{m-1}(x_{j+1}-x_{j})^{2}
\bigl(|g''(x_{j})|+|g''(x_{j+1})|\bigr).
\end{align}
\end{proposition}
\begin{proof}
From Corollary \ref{cor:34}, we have
\begin{align*}
\left|\int_{x_{j}}^{x_{j+1}}g(x)dx-(x_{j+1}-x_{j})g\left(\frac{x_{j}+x_{j+1}}{2}\right)\right|
\leq \min\{\delta_{1},\delta_{2}\}\,(x_{j+1}-x_{j})^{2}\bigl(|g''(x_{j})|+|g''(x_{j+1})|\bigr)
\end{align*}
Summing over $j$ from $0$ to $m-1$ and taking into account that $|g''|$ is convex, we
obtain, by the triangle inequality, that
\begin{align*}
\left|\int_{a}^{b}g(x)dx-T(g,d)\right|&=\left|\sum_{j=0}^{m-1}
\left[\int_{x_{j}}^{x_{j+1}}g(x)dx-(x_{j+1}-x_{j})g\left(\frac{x_{j}+x_{j+1}}{2}\right)\right]\right|
\\
&\leq \sum_{j=0}^{m-1}\left|
\int_{x_{j}}^{x_{j+1}}g(x)dx-(x_{j+1}-x_{j})g\left(\frac{x_{j}+x_{j+1}}{2}\right)\right|
\\
&\leq \min\{\delta_{1},\delta_{2}\}\,\sum_{j=0}^{m-1}(x_{j+1}-x_{j})^{2}
\bigl(|g''(x_{j})|+|g''(x_{j+1})|\bigr).
\end{align*}
This ends the proof.
\end{proof}
\subsection{Modified Bessel functions}
Let the function $\mathcal{I}_{p}:\mathbb{R}\to[1,\infty)$ be defined by
\begin{align*}
\mathcal{I}_{p}(x)=2^{p}\Gamma(p+1)x^{-v}I_{p}(x),\quad x\in\mathbb{R}.
\end{align*}
For this we recall the modified Bessel function of the first kind $\mathcal{I}_{p}$ which is defined 
as \cite{25}:
\begin{align*}
\mathcal{I}_{p}(x)&=\sum_{n\geq 0}\frac{\left(\frac{x}{2}\right)^{p+2n}}{n!\Gamma(p+n+1)}.
\end{align*} 
The first and the $n$th order derivative formula of $\mathcal{I}_{p}(x)$ are, respectively, given by \cite{26}:
\begin{align}\label{eq:prop61}
\mathcal{I}'_{p}(x)&=\frac{x}{2(p+1)}\mathcal{I}_{p+1}(x),
\\
\frac{\partial^{n}\mathcal{I}_{p}(x)}{\partial x^{n}}&=2^{n-2p}\sqrt{\pi}x^{p-n}\Gamma(p+1)\,
_{2}F_{3}\left(\frac{p+1}{2},\frac{p+2}{2};\frac{p+1-n}{2},\frac{p+2-n}{2},p+1;\frac{x^2}{4}\right),
\label{eq:prop62}
\end{align}
where $_{2}F_{3}\left(\cdot,\cdot;\cdot,\cdot,\cdot;\cdot\right)$ is the hypergeometric function defined 
by \cite{26}:
\begin{align}\label{eq:prop63}
&_{2}F_{3}\left(\frac{p+1}{2},\frac{p+2}{2};\frac{p+1-n}{2},\frac{p+2-n}{2},p+1;\frac{x^2}{4}\right)
=\sum_{k=0}^{\infty}\frac{\left(\frac{p+1}{2}\right)_{k}\left(\frac{p+2}{2}\right)_{k}}
{\left(\frac{p-2}{2}\right)_{k}\left(\frac{p-1}{2}\right)_{k}(p+1)_{k}}\frac{x^{2k}}{4^{k}\,(k)!},
\end{align}
where, for some parameter $\nu$, the Pochhammer symbol $(\nu)_{k}$ is defined as
\begin{align*}
(\nu)_{0}=1,\quad (\nu)_{k}=\nu(\nu+1)\cdots(\nu+k-1), \quad k=1,2,...
\end{align*}
\begin{proposition}
Let $u,v\in\mathbb{R}$ with $0<u<v$, then for each $p>-1$ we have
\begin{align}\label{eq:prop64}
&\left|\frac{\mathcal{I}_{p}(v)-\mathcal{I}_{p}(u)}{v-u}
-\frac{a+b}{4(p+1)}\mathcal{I}_{p+1}\left(\frac{u+v}{2}\right)\right|
\leq (v-u)^{2}\min\{\delta_{1},\delta_{2}\}\, 2^{3-2p}\sqrt{\pi}\Gamma(p+1)
\nonumber \\
&\times\Biggl(\left|a\right|^{p-3}
\left|\,_{2}F_{3}\left(\frac{p+1}{2},\frac{p+2}{2};\frac{p+1-n}{2},\frac{p+2-n}{2},p+1;\frac{a^2}{4}\right)
\right|
\nonumber \\
&+\left|b\right|^{p-3}
\left|\,_{2}F_{3}\left(\frac{p+1}{2},\frac{p+2}{2};\frac{p+1-n}{2},\frac{p+2-n}{2},p+1;\frac{b^2}{4}\right)
\right|\Biggr).
\end{align}
\end{proposition}
\begin{proof}
Let $g(x)=\mathcal{I}'_{p}(x)$. Note that the function $x\mapsto\mathcal{I}'''_{p}(x)$ is convex on the 
interval $[0,\infty)$ for each $p>-1$. Using Corollary \ref{cor:34} and \eqref{eq:prop61}--\eqref{eq:prop62}, 
we obtain the desired inequality \eqref{eq:prop64} immediately.
\end{proof}
\section{Conclusion}
In this paper, we established some new integral inequalities of midpoint type for convex functions with 
respect to increasing functions involving Riemann--Liouville fractional integrals. It can be noted from 
Corollary \ref{cor:31}--\ref{cor:33} that our results are a generalization of all obtained results in
\cite{21,22,23}.

\end{document}